\documentclass[a4paper,10pt,openright]{article}

\usepackage[english]{babel}
\usepackage[T1]{fontenc}
\usepackage{indentfirst}
\usepackage[utf8]{inputenc}
\usepackage{enumerate}

\usepackage{amsfonts}
\usepackage{amssymb}
\usepackage{amsmath}
\usepackage{latexsym}
\usepackage{amsthm}
\usepackage{amstext}
\usepackage{xfrac}

\usepackage{framed}

\usepackage{graphicx}
\usepackage{newlfont}
\usepackage{stmaryrd}
\usepackage{setspace}
\usepackage{color}
\usepackage{url}

\newcommand{\K}{\mathbb{K}}
\newcommand{\C}{\mathbb{C}}

\newcommand{\Q}{\mathbb{Q}}
\newcommand{\Z}{\mathbb{Z}}
\newcommand{\N}{\mathbb{N}}
\newcommand{\HP}{H \! P}

\newtheoremstyle{theorem}
{12pt} 
{12pt} 
{\slshape 

} 
{} 
{\bfseries} 
{} 
{ } 
{} 

\newtheoremstyle{definition}
{12pt} 
{12pt} 
{\slshape 

} 
{} 
{\bfseries} 
{} 
{ } 
{} 

\newtheoremstyle{algorithm}
{10pt} 
{10pt} 
{\slshape 

} 
{} 
{\bfseries} 
{} 
{ } 
{} 

\theoremstyle{theorem}
\newtheorem{theorem}{Theorem}

\newtheorem{proposition}[theorem]{Proposition}
\newtheorem{lemma}[theorem]{Lemma}

\theoremstyle{definition}
\newtheorem{definition}[theorem]{Definition}

\theoremstyle{remark}
\newtheorem{remark}[theorem]{Remark}
\newtheorem{example}[theorem]{Example}

\theoremstyle{algorithm}

\usepackage[dvipsnames]{xcolor}
\definecolor{BlueCarisma}{HTML}{004b84}
\definecolor{RosinoCarisma}{HTML}{dadada}

\definecolor{grad0}{rgb}{0 0.50002 1 }
\definecolor{grad1}{rgb}{0.00002 0 1}
\definecolor{grad2}{rgb}{0.24998 0 1}
\definecolor{grad3}{rgb}{0.5 0 1}
\definecolor{grad4}{rgb}{0.75002 0 1}
\definecolor{grad5}{rgb}{0.99998 0 1}
\definecolor{grad6}{rgb}{1 0 0.75}
\definecolor{grad7}{rgb}{1 0 0.49998}
\definecolor{grad8}{rgb}{1 0 0.25002}
\definecolor{grad9}{rgb}{1 0 0 }
\definecolor{grad10}{rgb}{1 0.1002 0}
\definecolor{grad11}{rgb}{1 0.26998 0}
\definecolor{grad12}{rgb}{1 0.4 0}
\definecolor{grad13}{rgb}{1 0.6 0}
\definecolor{grad14}{rgb}{1 0.8 0}
\definecolor{grad15}{rgb}{1 0.9 0}

\definecolor{grad_0}{cmyk}{0.7998 1 0 0.4}
\definecolor{grad_1}{cmyk}{0.7 0.6 0 0}
\definecolor{grad_2}{cmyk}{0.75002 0 1 0.1}
\definecolor{grad_3}{cmyk}{0.34998 0 1 0.1}

\title{A partial characterization of Hilbert quasi-polynomials in the non-standard case}
\date{\today}
\author{Massimo Caboara$^{*1}$, Carla Mascia$^{*2}$ \\ \small{$*1$ Department of Mathematics, University of Pisa, Italy} \\ \small{$*2$ Department of Mathematics, University of Trento, Italy} \\ \small{Email addresses: massimo.caboara@unipi.it, carla.mascia@unitn.it}}

\begin{document}

\maketitle

\begin{center}
\textbf{Abstract}
\end{center}

The Hilbert function, its generating function and the Hilbert polynomial of a
graded ring $\K[x_1,\ldots,x_k]$ have been extensively studied since the famous
paper of Hilbert: \emph{Ueber die Theorie der algebraischen Formen}
\cite{hilbert1890theorie}. In particular, the coefficients and the degree of the
Hilbert polynomial play an important role in Algebraic Geometry. 
If the ring grading is non-standard, then its Hilbert function 
is not eventually equal to a polynomial but to a quasi-polynomial.
It turns out that a Hilbert quasi-polynomial $P$ of degree $n$  splits into a
polynomial $S$ of degree $n$ and a lower degree quasi-polynomial $T$. 
We have completely determined the degree of $T$ and the first few
coefficients of $P$. Moreover, the quasi-polynomial $T$ has a periodic
structure that we have described.
We have also implemented a software to compute effectively the Hilbert
quasi-polynomial for any ring $\K[x_1,\ldots,x_k]/I$.

\medskip



\section{Introduction}
In this section, we provide a very short overview of Hilbert function, Hilbert-Poincar\'e series and Hilbert polynomial of graded rings, omitting details and proofs, which can be found in any introductory commutative algebra book (i.e. \cite{kreuzer2005computational2}).
From now on, $\K$ will be a field and $(R/I,W)$ stands for the polynomial ring $R/I$, where $R:=\K[x_1,\dots, x_k]$  is a quotient ring graded by $W := [d_1, \dots, d_k] \in \N^k_+$ and $I$ is a $W$-homogeneous ideal of $R$, that is 
\[
\deg(x_1^{\alpha_1}\cdots x_k^{\alpha_k}) := d_1 \alpha_1 + \cdots + d_k \alpha_k
\]
with the usual definition for the $n$-graded component of $R/I$.
Due to Macaulay's lemma, we can suppose that $I$ is a monomial ideal.
Recall that if $W =[1, \dots, 1]$ the grading is called \textbf{standard}. 

\

The \textbf{Hilbert function} $H_{R/I}^W : \N \rightarrow \N$ of $(R/I,W)$ is defined by 
\[
H_{R/I}^W (n) := \dim_{\K}((R/I)_n)
\]
and the Hilbert-Poincar\'e series of $(R/I,W)$ is given by
\[
\HP_{R/I}^W (t) := \sum_{n \in \N} H_{R/I}^W(n) t^n \in \N \llbracket t \rrbracket
\]

When the grading given by $W$ is clear from the contest, we denote respectively the Hilbert function and the Hilbert-Poincar\'e series of $(R/I,W)$ by $H_{R/I}$ and $\HP_{R/I}$.
The following two well-known results characterize the Hilbert-Poincar\'e series and the Hilbert function, but the latter holds only for standard grading.

\begin{theorem}[Hilbert-Serre] 

The Hilbert-Poincar\'e series of $(R/I,W)$ is a rational function, that is for suitable  $h(t) \in \Z[t]$ we have that
\[
\HP_{R/I}(t) = \frac{h(t)}{\prod_{i=1}^k( 1-t^{ d_i})}\in\Z\llbracket t \rrbracket
\]
\end{theorem}
\noindent
The polynomial $h(t)$ which appears in the denominator of $\HP_{R/I}$
is called \emph{h-vector} and we denote it by $< I >$.

\begin{definition}
Let $R$ be a polynomial standard graded ring and $I$ a homogeneous ideal of $R$. 
Then there exists a polynomial $\ P_{R/I}(x) \in \Q[x]$ such that 
\[
H_{R/I}(n) = P_{R/I}(n)\ \forall\ n \gg 0
\]
This polynomial is called \textbf{Hilbert polynomial of $R/I$}.
\end{definition}

\section{Hilbert quasi-polynomials} 
 
As we will see, in the case of non-standard grading, the Hilbert function of $(R/I,W)$ is definitely equal to a quasi-polynomial $P_{R/I}^W$ instead of a polynomial. The main aim of this section is to investigate the structure of $P_{R/I}^W$, such as its degree, its leading coefficient and proprieties of its coefficients.

\subsection{Existence of the Hilbert quasi-polynomials}

We recall that a function $f \colon \mathbb{N} \to \mathbb{N}$ is a (rational) \emph{quasi-polynomial of period s} if there exists a set of $s$ polynomials $\{p_0, \dots, p_{s-1}\}$ in $ \Q[x]$ such that we have $f(n) = p_i(n)$ when $n \equiv i \bmod s$. Let $d:= lcm(d_1,\dots, d_k)$. We are going to show that the Hilbert function of $(R/I, W)$ is definitely equal to a quasi-polynomial of period $d$. 

\begin{proposition} \label{def hilbert quasi-poly}
Let $(R/I,W)$ be as above. There exists a unique quasi-polynomial $P_{R/I}^W:=\{P_0, \dots, P_{d-1}\}$ of period $d$ such that $H_{R/I}(n) = P_{R/I}^W(n)$ for all $n\gg 0$, that is  
$$
H_{R/I}(n) = P_i(n) \qquad \forall i \equiv n\mod d  \quad \text{and} \quad \forall n\gg 0
$$
$P_{R/I}^W$ is called  the \textbf{Hilbert quasi-polynomial associated to (R/I,W)}.
\end{proposition}
\begin{proof}
From the Hilbert-Serre theorem, we know that the Hilbert-Poincar\'e series of $R/I$ can be written as a rational function 
\[
\HP_{R/I}(t) = \frac{h(t)}{\prod_{i=1}^{k} (1-t^{d_i})} \in \Z[|t|]
\]
Let $\zeta$ be a primitive dth root of unity, so we can write
\begin{equation*}
\prod_{i=1}^k (1-t^{d_i}) = \prod_{j=0}^{d-1} (1- \zeta^j t)^{\alpha_j}, \quad \text{for some $\alpha_j \in \N$ such that} \; \sum_{j=0}^{d-1}  \alpha_j = \sum_{i=1}^k d_i 
\end{equation*}
By using this relation and partial fractions, we have
\begin{equation}\label{eq HP_R/I}
\HP_{R/I}(t) = \frac{h(t)}{\prod_{j=0}^{d-1} (1- \zeta^j t)^{\alpha_j}} = \sum_{j=0}^{d-1} \frac{Q_j(t)}{(1- \zeta^j t)^{\alpha_j}}
\end{equation}
where $Q_j(t) \in \C[t]$ is a polynomial of degree $n_j$ for all $j=0, \dots, d-1$, namely $Q_j(t) := \sum_{h=0}^{n_j} a_{jh} t^h$.

\medskip

We are going to consider individually any addend in (\ref{eq HP_R/I}), so 
\begin{align*}
\frac{Q_j(t)}{(1- \zeta^j t)^{\alpha_j}} &= \left( \sum_{h=0}^{n_j} a_{jh} t^h \right) \cdot \frac{1}{(1- \zeta^j t)^{\alpha_j}} \\
&= \left( \sum_{h=0}^{n_j} a_{jh} t^h \right) \cdot \left( \sum_{n \geq 0} \binom{ n + \alpha_j -1}{n} \zeta ^{jn} t^n \right) \\
&= \sum_{n\geq 0} \left[ \sum_{h=0}^{n_j} a_{jh}\binom{n+ \alpha_j -h-1}{n-h} \zeta^{j(n-h)} \right] t^n = \sum_{n\geq 0} b_{jn} t^n
\end{align*}
where 
\[b_{jn} := \sum_{h=0}^{n_j} a_{jh}\binom{n+ \alpha_j -h-1}{n-h} \zeta^{j(n-h)} \\
\]
We can rewrite the Hilbert-Poincar\'e series in the following way:
\begin{align*}
\HP_{R/I}(t) = \sum_{j=0}^{d-1} \frac{Q_j(t)}{(1- \zeta^j t)^{\alpha_j}} = \sum_{j=0}^{d-1} \sum_{n \geq 0} b_{jn} t^n = \sum_{n \geq 0} \left( \sum_{j=0}^{d-1} b_{jn} \right) t^n
\end{align*}
and, therefore, for all $n \geq \underset {j}{\max} \{n_j\}$, we have
\begin{align*}
H_{R/I} (n) = \sum_{j=0}^{d-1}  b_{jn}= \sum_{j=0}^{d-1} \sum_{h=0}^{n_j} a_{jh}\binom{n+ \alpha_j -h-1}{n-h} \zeta^{j(n-h)} = \sum_{j=0}^{d-1} S_j (n) \zeta^{jn}
\end{align*}
where 
\[
S_j (n) := \sum_{h=0}^{n_j} a_{jh}\binom{n+ \alpha_j -h-1}{n-h} \zeta^{-jh}\\
\]
So, given $n \geq \underset {j}{\max} \{n_j\}$, let $i$  be an integer such that $0 \leq i \leq d-1$ and $i \equiv n \bmod d$, the Hilbert function evaluated at $n$ is equal to the evaluation at $n$ of the polynomial 
\begin{align*}
P_i(x) := S_0(x) + S_1(x)\zeta^i + \dots + S_{d-1}(x) \zeta^{(d-1)i}
\end{align*}
\end{proof}

By the proof, we find out that $H_{R/I}(n) = P_{R/I}^W(n)$ for all $n \geq \underset {j}{\max} \{n_j\}$, where $n_j$ are the degrees of the polynomials $Q_j(t)$ which appear in Equation ($\ref{eq HP_R/I}$) of the Hilbert-Poincar\'e series.

\begin{remark}
Proposition \ref{def hilbert quasi-poly} asserts that the Hilbert quasi-polynomial consists of $d$ polynomials, but it doesn't assure that they are all distinct. Actually, it can happen that some of them are equal to each other, a trivial example is given  for $gcd(d_1, \dots, d_k) \neq 1$, as we will see.
Remark that the proof of proposition yields a way to compute the Hilbert quasi-polynomial of $(R/I,W)$, when the Hilbert-Poincar\'e series is known.
\end{remark}


%
%

\begin{remark}
All the polynomials of the Hilbert quasi-polynomial $P_{R/I}^W$ have rational coefficients. In fact, we recall that if a polynomial $P(x) \in \C[x]$ of degree $n$ is such that $P(x_i) \in \Z$ for some $x_0, \dots, x_n \in \Z$, then $P(x) \in \Q[x]$. By definition, $P_i(n) \in \N$ for all sufficiently large $n \in \N$ such that $i \equiv n \mod d$.
\end{remark}

Due to the following result, we can restrict our study to the case $(R, W)$ where $W$ is such that $d=1$. 

\begin{proposition}\label{prop6}
Let $W' := a \cdot W = [d_1', \dots, d_k']$ for 
some $a \in \N_+$ and let $\displaystyle \HP_{R/I}(t) = \frac{\sum_{j=0}^r a_jt^j}{\prod_{i=1}^k(1-t^{d_i})}$. Then it holds:
\begin{enumerate}

\item 
$P_{R/I}^W(n) = \sum_{j=0}^r a_j P_R^W(n-j)\ \forall\ n\gg0$

\item 

$P_R^{W'}=\{P_0', \dots, P_{ad-1}'\}$ is such that
\[
P_i'(x) =\begin{cases}
0&\text{ if } a \nmid i \; \\
P_{\frac{i}{a}}\left(\frac{x}{a}\right)& \text{ if } a \mid i \;  
\end{cases}
\]

\end{enumerate}
\end{proposition}

%
\begin{proof}
\begin{enumerate}[(i)]
\item
Since \[
\HP_R(t) = \sum_{n \geq 0} H_R(n)t^n = \frac{1}{\prod_{i=1}^k (1-t^{d_i})}
\]
we have
\begin{align*}
\HP_{R/I}(t) &= \frac{h(t)}{\prod_{i=1}^k (1-t^{d_i})} = \left( \sum_{n \geq 0} H_R(n)t^n \right) \left( \sum_{j=0}^r a_jt^j \right) \\ &= \sum_{n \geq 0} \left( \sum_{j=0}^r a_j H_R(n-j) \right) t^n
\end{align*}
with the convention $H_R(n) = 0$ for all $n < 0$. Therefore, for all $n \geq r$, $H_{R/I}(n) = \sum_{j=0}^r a_j H_R(n-j)$ . Given that $H_R(n) = P_R^W(n)$ for all sufficiently large $n$, then $H_R(n-j) = P_R^W(n-j)$ for all sufficiently large $n$ and so $H_{R/I}(n) = \sum_{j=0}^r a_j P_R^W(n-j)$ for all sufficiently large $n$. The statement follows.
\item Let $i$ be a positive integer such that $a$ doesn't divide $i$. We observe that $H_R^{W'}(n) = 0$ for all $n \in \N$ such that $ n \equiv i \bmod ad$. Indeed, assume that $H_R^{W'}(n) \neq 0$ for some $n \equiv i \bmod ad$. Then there exist $a_1, \dots, a_k \in \N$ such that $a_1d_1' + \dots + a_k d_k' = n$. By hypothesis, $a$ divides $d_1',\dots,d_k'$, hence $a$ divides $n$ and so $a$ divides $i$, which is a contradiction. 
Since $P'_i(n) = H_R^{W'}(n)$ for all sufficiently large $n$ such that $n \equiv i \bmod a$, then $P_i'$ has an infinitive number of roots and so $P_i' = 0$.

Let $i$ be a positive integer such that $a$ divides $i$. Recall that for all sufficiently large $n$ such that $n \equiv i \bmod d$, we have
\[
P_i(n) = H_R^{W}(n)= \# \{(a_1,\dots, a_k) \in \N^k | a_1d_1 + \dots a_kd_k = n \}
\]
Then, 
\begin{align*}
P_i' (n)= H_R^{W'}(n) &= \#\{(b_1, \dots, b_k) \in \N^k \mid ab_1d_1 + \dots + ab_kd_k = n\} \\ &= \# \{(b_1, \dots, b_k) \in \N^k \mid b_1d_1 + \dots + b_kd_k = {n}/{a} \} \\
&= H_R^{W}\left(\frac{n}{a}\right) = P_{\frac{i}{a}} \left(\frac{n}{a}\right)
\end{align*}
\end{enumerate}
\end{proof}


To avoid endlessly repeating these hypothesis, we shall use $W$ to denote a weight vector $[d_1, \dots, d_k]$ such that $gcd(d_1,\dots, d_k) = 1$, throughout the remainder of this section.

\subsection{Degree and leading coefficient of the Hilbert quasi-polynomial}
\sectionmark{Degree and leading coefficient}
In this subsection, we investigate some proprieties of the Hilbert quasi-polynomial of $(R,W)$. First of all, we introduce the following notation which will help us to rewrite the denominator of the Hilbert-Poincar\'e series of $(R,W)$ in a suitable way.
Given $d_1,\dots,d_k \in \N_+$, with $d:= lcm(d_1, \dots, d_k)$, we define

\begin{itemize}
\item $\delta:= max\{ |I|  \mid gcd(d_i)_{i \in I} \neq 1 \; \text{and} \; I \subseteq \{1,\dots, k\} \}$,
\item $\hat{d}_s := \frac{d}{d_s}$,
\item $M_s:= \{\hat{d}_s, 2\hat{d}_s, \dots, (d_s -1) \hat{d}_s\}\;$ for all $s=1,\dots,k$,
\item $T_r := \displaystyle\bigcup_{\substack{J \subseteq \{1,\dots,k\} \\ |J| = r}} \left( \bigcap _{s \in J} M_s\right)\;\;$ for all $r=1,\dots,k$. 
\end{itemize}

Observe that $T_r$ is the set of the elements $p \in \bigcup_{i=1}^k M_i$ which belong to exactly $r$ sets $M_i$.

Now we'll analyse the denominator of the Hilbert-Poincar\'e series of $(R,W)$, which we denote by $g(t)$. Let $\zeta := \zeta_d$ be a primitive dth root of unity. We remark that $\zeta_{d_j} = \zeta^{\hat{d}_j}$ for all $j=1,\dots,k$. Then we can write

\begin{equation} \label{eq g(t)}
\begin{split}
g(t) & = \prod_{i=1}^k (1-t^{d_i}) = (1-t)^k \ \prod_{i=1}^k \left[ (1-\zeta^{\hat{d}_i}t) \cdots (1-\zeta^{(d_i-1) \hat{d}_i}t) \right] \\
& = (1-t)^k  \prod_{j \in M_1} (1-\zeta^j t) \cdots \prod_{j \in M_k} (1-\zeta^j t)\\
& = (1-t)^k  \prod_{j \in T_k} (1-\zeta^j t)^k \cdots \prod_{j \in T_1} (1-\zeta^j t) \\
\end{split}
\end{equation}


We are now going to show some lemmas concerning the sets $M_s$ and $T_r$.

\begin{lemma}\label{lemma1}
Let $d_1,\dots,d_k \in \N_+$. For any subset $\{j_1,\dots, j_r \} \subseteq \{1,\dots, k\}$, it holds

\begin{align*}
M_{j_1} \cap \cdots \cap M_{j_r} =  & \left\lbrace \frac{d}{gcd(d_{j_1},\dots,d_{j_r})}, 2 \cdot \frac{d}{gcd(d_{j_1},\dots,d_{j_r})},\dots  \right. \\ & \left. \quad \dots, (gcd(d_{j_1},\dots,d_{j_r})-1) \cdot\frac{d}{gcd(d_{j_1},\dots,d_{j_r})} \right\rbrace
\end{align*}
\end{lemma}

\begin{proof}
We use induction on $r$, the number of intersected subsets.
The case $r = 1$ is trivial. Suppose $r \geq 2$ and consider, without loss of generality, the subsets $M_1, \dots, M_r$ instead of $M_{j_1},\dots, M_{j_r}$. By induction hypothesis, we have
\begin{align*}
M_{1} \cap \cdots \cap M_{r-1} = & \left\lbrace\frac{d}{gcd(d_{1},\dots,d_{r-1})}, 2 \cdot \frac{d}{gcd(d_{1},\dots,d_{r-1})},\dots \right. \\ & \left. \dots, (gcd(d_{1},\dots,d_{r-1})-1) \cdot\frac{d}{gcd(d_{1},\dots,d_{r-1})} \right\rbrace
\end{align*}
We observe that $M_1 \cap \dots \cap M_{r-1}$ has the same structure of $M_a$, where $a$ is equal to $gcd (d_1,\dots,d_{r-1})$ (without loss of generality we can suppose that $a$ is one of the given integers). Therefore $M_1 \cap \dots \cap M_{r} = M_a \cap M_{r}$ and, by using again the induction hypothesis, we obtain
\begin{align*}
M_1 \cap \dots \cap M_{r}= \left\lbrace \frac{d}{gcd(a,d_{r})}, 2 \cdot \frac{d}{gcd(a,d_{r})},\dots, (gcd(a,d_{r})-1) \cdot\frac{d}{gcd(a,d_{r})} \right\rbrace
\end{align*}
and since $a= gcd (d_1,\dots,d_{r-1})$, we are done.
\end{proof}

\begin{lemma}\label{lemma2}
Let $d_1,\dots,d_k \in \N_+$. Then $T_{\delta +1} = \dots = T_k = \emptyset$.
\end{lemma}
\begin{proof}
We assume that $T_h \neq \emptyset$ for some $h=\delta+1, \dots, k$. Consider $p \in T_h$ and suppose, without loss of generality, that $p \in M_1 \cap \dots \cap M_h$. By lemma \ref{lemma1}, $p = p_1 \cdot \frac{d}{gcd(d_1, \dots, d_h)}$ with $1 \leq p_1 \leq gcd(d_1, \dots, d_h) -1$. Since $\delta< h $, we must have $gcd(d_1, \dots, d_h) = 1$. Therefore $p$ is a multiple of $d$, which is a contradiction because of $p < d$.
\end{proof}

By lemma \ref{lemma2}, in Equation (\ref{eq g(t)}) the set $T_{\delta +1}, \dots, T_k$ are empty, so ultimately we have

\begin{equation} \label{g(t)}
g(t) = \prod_{i=1}^k (1-t^{d_i}) = (1-t)^k  \left[ \prod_{j \in T_{\delta}} (1-\zeta^j t)^{\delta}\right] \cdots \left[ \prod_{j \in T_1} (1-\zeta^j t)\right]
\end{equation}

and we obtain the following useful expression for the Hilbert-Poincar\'e series
\begin{equation}\label{eq HP}
\begin{split}
\HP_{R} (t) = \frac{1}{g(t)} = \frac{Q_0(t)}{(1-t)^k} + \sum_{j \in T_{\delta}} \frac{Q_j^{(\delta)}(t)}{(1-\zeta^j t)^{\delta}} + \dots + \sum_{j \in T_1} \frac{Q_j^{(1)}(t)}{(1-\zeta^j t)}\\
\end{split}
\end{equation}

\begin{remark}\label{remark1}
We have already seen in the proof of Proposition \ref{def hilbert quasi-poly} that any addend $\frac{Q_j(t)}{(1- \zeta^j t)^{\alpha_j}}$ of $\HP_R (t)$, with $Q_j(t) = \sum _{h=0}^{n_j} a_{jh}t^h$, can be written in the following way

\begin{equation} \label{eq1}
\frac{Q_j(t)}{(1- \zeta^j t)^{\alpha_j}} = \sum_{n\geq 0} \left[ \sum_{h=0}^{n_j} a_{jh}\binom{n+ \alpha_j -h-1}{\alpha_j -1} \zeta^{j(n-h)} \right] t^n 
\end{equation}
\

\noindent
Then the rational function $\frac{Q_j(t)}{(1- \zeta^j t)^{\alpha_j}}$ corresponds to a power series in $t$ where the nth coefficient depends on n in a polynomial way through the binomial $\binom{n+ \alpha_j -h-1}{\alpha_j -1}$ and its degree in $n$ is $\alpha_j-1$.
\end{remark}

We consider the equation (\ref{eq HP}) of the Hilbert-Poincar\'e series and, by using partial fractions, we get
\begin{equation}\label{eq2 HP}
\begin{split}
\HP_R(t) = \frac{1}{g(t)} &= \frac{A_1}{1-t} + \frac{A_2}{(1-t)^2} + \dots + \frac{A_k}{(1-t)^k}+ \\ &+ \sum_{j \in T_\delta} \left[ \frac{B_{j,1}^{(\delta)}}{1-\zeta^jt} + \dots + \frac{B_{j,\delta}^{(\delta)}}{(1-\zeta^j t)^ \delta} \right] + \dots + \sum_{j \in T_1} \left[ \frac{B_{j,1}^{(1)}}{1-\zeta^j t} \right]
\end{split}
\end{equation}
for some $A_i, B_{j,h}^{(r)} \in \Q$.

\begin{remark}
By multiplying both sides of Equation (\ref{eq2 HP}) by $g(t)$, we get
\begin{equation}\label{1 = ...}
\begin{split}
&1 = \sum_{t=1}^k \left[ A_t (1-t)^{k-t} \cdot \prod_{j \in T_\delta} (1-\zeta^j t)^\delta \cdots \prod_{j \in T_1} (1-\zeta^j t) \right] + \\
&+\sum_{j \in T_\delta} \left\lbrace (1-t)^k \cdot \prod_{i=1}^{\delta -1} \left[ \prod_{r \in T_i} (1-\zeta^jt)^i \right] \cdot \sum_{m=1}^k \left[  B_{j,m}^{(\delta)} \cdot \prod_{\substack{s \in T_{\delta} \\ s \neq  j}} (1-\zeta^s t )^{\delta - m}   \right] \right\rbrace + \\
&\; +\\
& \; \;\vdots \\
&\; + \\
&+\sum_{j \in T_1} \left[ B_{j,1}^{(1)}(1-t)^k \cdot \prod_{r \in T_\delta} (1-\zeta^r t)^{\delta} \cdots \prod_{r \in T_2} (1-\zeta^r)^2 \prod_{\substack{s \in T_1 \\ s \neq j}}(1-\zeta^st)\right]
\end{split}
\end{equation}
\end{remark}

\begin{remark}
By Equation (\ref{eq2 HP}), we have 
\begin{equation} \label{Q_0(t)}
Q_0(t) = A_1 (1-t)^{k-1} + A_2 (1-t)^{k-2} + \dots + A_k
\end{equation}
\end{remark}

\begin{proposition}
Let $(R,W)$ and $P_R^W$ be as above. The degree of $P_i$ is equal to $k-1$ for all $i=0, \dots, d-1$ and its leading coefficient $lc(P_i)$ is such that 
\[
lc(P_i) = \frac{1}{(k-1)!\prod_{i=1}^k d_i}
\]
\end{proposition}
\begin{proof}
Since $\delta < k$ by hypothesis and thanks to Equation (\ref{eq HP}) and remark \ref{remark1}, it suffices to analyse the term
\[
\frac{Q_0(t)}{(1-t)^k} = \sum_{n \geq 0} \left[ \sum_{h=0}^{k-1} a_h \binom{n-h+k-1}{k-1} \right] t^n
\]
\

where $a_h := a_{0,h}$ for all $h = 0, \dots, k-1$.

We observe that $\sum_{h=0}^{k-1} a_h \binom{n-h+k-1}{k-1} = S_0(n)$ and then the degree of $P_i$ is less or equal to $k-1$. It remains to show that the coefficient of $n^{k-1}$ is not 0. In particular, we wish to show that it is equal to $\frac{1}{(k-1)!\prod_{i=1}^k d_i}$.

Suppose $n \geq k-1$, then the nth coefficient of $\HP_R(t)$ is equal to 
\begin{equation*}
\begin{split}
S_0(n) &= \sum_{h=0}^{k-1} a_h \binom{n-h+k-1}{k-1} \\
&= \sum_{h=0}^{k-1} a_h \frac{(n-h+k-1)(n-h+k-2)\cdots(n-h+1)}{(k-1)!} \\
&= \sum_{h=0}^{k-1} a_h \cdot \frac{n^{k-1}}{(k-1)!} + \text{terms of degree in n lower than } (k-1)
\end{split}
\end{equation*}
Then the leading coefficient of $S_0(n)$ is
\begin{equation}\label{last_step_leadcoeff}
\frac{1}{(k-1)!} \sum_{h=0}^{k-1} a_h = \frac{Q_0(1)}{(k-1)!} = \frac{A_k}{(k-1)!} = \frac{1}{(k-1)!\prod_{i=1}^k d_i}
\end{equation}
where the last equality in (\ref{last_step_leadcoeff}) follows by the Lemma \ref{formula_Ak}.
\end{proof}

\begin{lemma}\label{formula_Ak}
The constant $A_k$ which appears in Equation (\ref{eq2 HP}) and (\ref{Q_0(t)}) is such that
\[
A_k = \frac{1}{\prod_{i=1}^k d_i}.
\]
\end{lemma}

\begin{proof}
By evaluating Equation (\ref{1 = ...}) at $t = 1$, we get
\[
1 = A_k \cdot \prod_{j \in T_\delta} (1-\zeta^j)^{\delta}\cdots \prod_{j \in T_1} (1-\zeta^j)
\]
If we prove that $\prod_{j \in \delta} (1-\zeta^j)^{\delta} \cdots \prod_{j \in T_1} (1-\zeta^j)$ is equal to $\prod_{i=1}^k d_i$, we are done.  For this purpose, we consider the following two equations, the first of which derives from equation (\ref{g(t)}):
\begin{equation*}
\begin{split}
&\frac{g(t)}{(1-t)^k}=  \prod_{j \in T_\delta} (1-\zeta^jt)^{\delta} \cdots \prod_{j \in T_1} (1-\zeta^jt)  \\
&\frac{g(t)}{(1-t)^k}= \prod_{j=1}^k \frac{1-t^{d_j}}{1-t} = \prod_{j=1}^k (1+t+\dots+t^{d_j -1})
\end{split}
\end{equation*}
Then, it holds
\begin{equation} \label{g(t)/(1-t)^k}
\prod_{j \in T_\delta} (1-\zeta^jt)^{\delta} \cdots \prod_{j \in T_1} (1-\zeta^jt)=  \prod_{j=1}^k (1+t+\dots+t^{d_j -1})
\end{equation}
By evaluating both sides of Equation (\ref{g(t)/(1-t)^k}) at $t = 1$, the statement follows.
\end{proof}

Now, we'll give a degree bound for the Hilbert quasi-polynomial of \, $(R/I, W)$, when $I \neq (0)$.

\begin{proposition}\label{degree P case R/I}
Let $I \neq (0)$ a $W$-homogeneous ideal of $R$ and let $P_{R/I}^W$ the Hilbert quasi-polynomial of $(R/I,W)$. Then the degree of $P_i$ is less or equal to $k-2$.
\end{proposition}

In order to prove Proposition \ref{degree P case R/I}, we need the following results. Let $I \neq (0)$ be a homogeneous monomial ideal and consider the Hilbert-Poincar\'e series $\HP_{R/I}(t) = \frac{f(t)}{\prod_{i=1}^k(1-t^{d_i})} = \frac{f(t)}{g(t)}$ of $(R/I,W)$.
Let $f(t) := q(t)g(t) + r(t)$, where $q(t), r(t) \in \Q(t)$ and $\deg r < \deg g$. Then, we can write

\begin{equation}\label{eq2 HP_R/I}
\begin{split}
\HP_{R/I}(t) &= \frac{f(t)}{g(t)} = q(t) + \frac{r(t)}{g(t)} = q(t) + \left\lbrace \frac{C_1}{1-t} + \frac{C_2}{(1-t)^2} + \dots + \frac{C_k}{(1-t)^k}+ \right. \\ & \left.+ \sum_{j \in T_\delta} \left[ \frac{D_{j,1}^{(\delta)}}{1-\zeta^jt} + \dots + \frac{D_{j,\delta}^{(\delta)}}{(1-\zeta^j t)^ \delta} \right] + \dots + \sum_{j \in T_1} \left[ \frac{D_{j,1}^{(1)}}{1-\zeta^j t} \right] \right\rbrace
\end{split}
\end{equation}
for some $C_i, D_{j,h}^{(r)} \in \Q$.

\begin{lemma}\label{C_k in HP_R/I}
The constant $C_k$ which appears in Equation (\ref{eq2 HP_R/I}) is zero.
\end{lemma}

\begin{proof}  
By Equation (\ref{eq2 HP_R/I}), we get 

\begin{equation*}
\begin{split}
&r(t) = \sum_{t=1}^k \left[ C_t (1-t)^{k-t} \cdot \prod_{j \in T_\delta} (1-\zeta^j t)^\delta \cdots \prod_{j \in T_1} (1-\zeta^j t) \right] + \\
&+\sum_{j \in T_\delta} \left\lbrace (1-t)^k \cdot \prod_{i=1}^{\delta -1} \left[ \prod_{r \in T_i} (1-\zeta^jt)^i \right] \cdot \sum_{m=1}^k \left[ D_{j,m}^{(\delta)} \cdot \prod_{\substack{s \in T_{\delta} \\ s \neq  j}} (1-\zeta^s t )^{\delta - m}   \right] \right\rbrace + \\
&\;+\\
& \;\;\vdots \\
&\;+ \\
\end{split}
\end{equation*}
\begin{equation}\label{eq3 HP_R/I}
\begin{split}
&+\sum_{j \in T_1} \left[ D_{j,1}^{(1)}(1-t)^k \cdot \prod_{r \in T_\delta} (1-\zeta^r t)^{\delta} \cdots \prod_{r \in T_2} (1-\zeta^r)^2 \prod_{\substack{s \in T_1 \\ s \neq j}}(1-\zeta^st)\right]
\end{split}
\end{equation}
By evaluating Equation (\ref{eq3 HP_R/I}) at $ t= 1$, we have 
\[
r(1) = C_k \cdot \prod_{j \in T_{\delta}} (1-\zeta^j)^{\delta} \cdots  \prod_{j \in T_1} (1-\zeta^j) 
\]
We already know that $\prod_{j \in T_{\delta}} (1-\zeta^j)^{\delta} \cdots  \prod_{j \in T_1} (1-\zeta^j) = \prod_{i=1}^k d_i$, then we have to show that $r(1)=0$. Actually, it suffices to prove that $f(1)=0$, because $f(1) = q(1)g(1) + r(1)$ and $f(1) = g(1) = 0$ imply $r(1) = 0$.
We prove that $f(1) = 0$ by induction on the cardinality $s$ of a minimal set of generators of $I$.
Suppose $s=1$, since $I \neq (0)$ there exists $m \in I$ such that $I = (m)$. Then $f(t) = < I >= 1- t^{\deg_W(m)}$, so obviously $f(1)=0$.
If $s \geq 2$, let $\{m_1, \dots, m_s\}$ a minimal set of generators of I, then $I=(m_1, \dots, m_s)$ and, by using proprieties of the h-vector, we have
\begin{multline*}
f(t) = \;< (m_1,\dots,m_s) > \;=\\
<(m_1,\dots,m_{s-1})> - t^{\deg_W(m_s)} < (m_1,\dots,m_{s-1}) : (m_s)>
\end{multline*}

and we can conclude by induction hypothesis. 
\end{proof}

We can now demonstrate that for all $i=0,\dots,d-1$ it holds $\deg(P_i) \leq k-2$.

\begin{proof}[Proof of Proposition \ref{degree P case R/I}]
We consider the following expression for Hilbert-Poincar\'e series, which we have already seen in Equation (\ref{eq2 HP_R/I})
\begin{equation*}
\begin{split}
\HP_{R/I}(t) &= \frac{f(t)}{g(t)} = q(t) + \frac{r(t)}{g(t)} = q(t) + \left\lbrace \frac{C_1}{1-t} + \frac{C_2}{(1-t)^2} + \dots + \frac{C_k}{(1-t)^k}+ \right. \\ & \left.+ \sum_{j \in T_\delta} \left[ \frac{D_{j,1}^{(\delta)}}{1-\zeta^jt} + \dots + \frac{D_{j,\delta}^{(\delta)}}{(1-\zeta^j t)^ \delta} \right] + \dots + \sum_{j \in T_1} \left[ \frac{D_{j,1}^{(1)}}{1-\zeta^j t} \right] \right\rbrace
\end{split}
\end{equation*}
Since $\delta< k$, and by lemma \ref{C_k in HP_R/I}, which asserts that $C_k = 0$, we have that the maximum power of $t$ in the denominator of $\HP_{R/I}$ is at most $k-1$. Therefore, by recalling the construction of Hilbert quasi-polynomial done in the proof of proposition \ref{def hilbert quasi-poly}, we write
\[
\HP_{R/I} (t) = \sum_{j=0}^{d-1} \frac{Q_j(t)}{(1-\zeta^j t)^{\alpha_j}} 
\]
for some polynomials $Q_j (t)  := \sum_{h=0}^{n_j} a_{jh}t^h \in \C[t]$ of degree $n_j$ with $\alpha_j \leq k-1$ for all $j=0, \dots, d-1$. Then, for all $n \geq max_{\substack{j}} \{n_j\}$, we have
\[
H_{R/I}(n) = \sum_{j=0}^{d-1} S_j(n)\zeta^{jn}
\qquad \text{where} \; \;
S_j (n):= \sum_{h=0}^{n_j} \binom{n + \alpha_j -h -i}{\alpha_j -1} \zeta^{-jh}
\]
Observe that $S_j$ is a polynomial in $n$ of degree $\alpha_j -1$.
Finally, we get 
$$
P_i (x) := S_0(x) + S_1(x) \zeta^i + \dots + S_{d-1}(x)\zeta^{(d-1)i}
$$
and, by our previous considerations, every $S_j$ has degree less or equal to $k-2$ and the statement follows.

\end{proof}

%
%
%
%
%

\subsection{Proprieties of the coefficients of the Hilbert quasi-polynomial}
\sectionmark{Proprieties of the coefficients}
In this subsection, we present some proprieties of the coefficients of Hilbert quasi-polynomials. In particular, we show that a Hilbert quasi-polynomial of degree $n$ splits into a polynomial of degree $n$ and a lower degree quasi-polynomial which have a periodic structure that we will describe.

Before showing our results, we present an example.

\begin{example}\label{ex1}
Consider $R=\Q[x_1,\dots,x_5]$ graded by $W=[1,2,3,4,6]$. We have $k=5$ and $d=12$ and the Hilbert quasi-polynomial $P_R^W=\{P_0,\dots, P_{11}\}$ is 

{\small
\setlength{\tabcolsep}{2.5pt}
\begin{tabular}{lllllllllll}
$P_0(x)$&$=$&$\textcolor{blue}{1/3456} x^4$&$+$&$\textcolor{BlueCarisma}{1/108}x^3$&$+$&$\textcolor{teal}{5/48}x^2$&$+$&$\textcolor{grad_0}{1/2}x$&$+$&$\textcolor{grad0}{1}$ \\
$P_1(x)$&$=$&$\textcolor{blue}{1/3456}x^4$&$+$&$\textcolor{BlueCarisma}{1/108}x^3$&$+$&$\textcolor{olive}{19/192}x^2$&$+$&$\textcolor{grad_1}{43/108}x$&$+$&$\textcolor{grad1}{1705/3456}$ \\
$P_2(x)$&$=$&$\textcolor{blue}{1/3456}x^4$&$+$&$\textcolor{BlueCarisma}{1/108}x^3$&$+$&$\textcolor{teal}{5/48}x^2$&$+$&$\textcolor{grad_2}{25/54}x$&$+$&$\textcolor{grad2}{125/216}$ \\
$P_3(x)$&$=$&$\textcolor{blue}{1/3456}x^4$&$+$&$\textcolor{BlueCarisma}{1/108}x^3$&$+$&$\textcolor{olive}{19/192}x^2$&$+$&$\textcolor{grad_3}{5/12}x$&$+$&$\textcolor{grad3}{75/128}$ \\
$P_4(x)$&$=$&$\textcolor{blue}{1/3456}x^4$&$+$&$\textcolor{BlueCarisma}{1/108}x^3$&$+$&$\textcolor{teal}{5/48}x^2$&$+$&$\textcolor{darkgray}{13/27}x$&$+$&$\textcolor{grad4}{20/27}$ \\
$P_5(x)$&$=$&$\textcolor{blue}{1/3456}x^4$&$+$&$\textcolor{BlueCarisma}{1/108}x^3$&$+$&$\textcolor{olive}{19/192}x^2$&$+$&$\textcolor{gray}{41/108}x$&$+$&$\textcolor{grad5}{1001/3456}$ \\
$P_6(x)$&$=$&$\textcolor{blue}{1/3456}x^4$&$+$&$\textcolor{BlueCarisma}{1/108}x^3$&$+$&$\textcolor{teal}{5/48}x^2$&$+$&$\textcolor{grad_0}{1/2}x$&$+$&$\textcolor{grad6}{7/8}$ \\
$P_7(x)$&$=$&$\textcolor{blue}{1/3456}x^4$&$+$&$\textcolor{BlueCarisma}{1/108}x^3$&$+$&$\textcolor{olive}{19/192}x^2$&$+$&$\textcolor{grad_1}{43/108}x$&$+$&$\textcolor{grad7}{1705/3456}$ \\
$P_8(x)$&$=$&$\textcolor{blue}{1/3456}x^4$&$+$&$\textcolor{BlueCarisma}{1/108}x^3$&$+$&$\textcolor{teal}{5/48}x^2$&$+$&$\textcolor{grad_2}{25/54}x$&$+$&$\textcolor{grad8}{19/27}$ \\
$P_9(x)$&$=$&$\textcolor{blue}{1/3456}x^4$&$+$&$\textcolor{BlueCarisma}{1/108}x^3$&$+$&$\textcolor{olive}{19/192}x^2$&$+$&$\textcolor{grad_3}{5/12}x$&$+$&$\textcolor{grad9}{75/128}$ \\
$P_{10}(x)$&$=$&$\textcolor{blue}{1/3456}x^4$&$+$&$\textcolor{BlueCarisma}{1/108}x^3$&$+$&$\textcolor{teal}{5/48}x^2$&$+$&$\textcolor{darkgray}{13/27}x$&$+$&$\textcolor{grad10}{133/216}$\\
$P_{11}(x)$&$=$&$\textcolor{blue}{1/3456}x^4$&$+$&$\textcolor{BlueCarisma}{1/108}x^3$&$+$&$\textcolor{olive}{19/192}x^2$&$+$&$\textcolor{gray}{41/108}x$&$+$&$\textcolor{grad11}{1001/3456}$ \\[1pt]
\newline
\end{tabular}
}


We have 12 polynomials of degree 4 with leading coefficient equal to 1/3456, that's just what we expected.
As can be seen by comparing polynomials, the following facts hold:
\begin{itemize}
\item The coefficient of the term of degree 3 is the same for all polynomials.
\item The coefficient of the term of degree 2 has periodicity 2, i.e. 5/48 is the coefficient of the term of degree 2 of all $P_i(x)$ such that $i \equiv 0 \mod 2$ and 19/192 is the coefficient of the term of degree 2 of all $P_i(x)$ such that $i \equiv 1 \mod 2$.
\item The coefficient of the term of degree 1 has periodicity 6, i.e. 1/12 is the coefficient of the term of degree 1 of all $P_i(x)$ such that $i \equiv 0 \mod 6$, 43/108 is the coefficient of the term of degree 1 of all $P_i(x)$ such that $i \equiv 1 \mod 6$, 5/48 is the coefficient of the term of degree 1 of all $P_i(x)$ such that $i \equiv 2 \mod 6$ and so on.
\item The constant term seems not to have any kind of regularity.
\end{itemize}
\end{example}

The goal of next two results is to see if we can predict which and how coefficients change. We use $\delta$ to denote  $max\{ |I| \mid gcd(d_i)_{i \in I} \neq 1 \; \text{and} \; I \subseteq \{1,\dots, k\} \}$.

\begin{proposition}\label{prop link}
Let $(R,W)$ and $P_R^W$ be as above. Then 
\[
P_R^W(x) = Q(x) + R(x)
\]
where $Q(x) \in \Q[x]$ is a polynomial of degree $k-1$, whereas $R(x)$ is a quasi-polynomial with rational coefficients of degree $\delta -1$.
\end{proposition}

\begin{proof}
Recall that due to Lemma \ref{lemma2} we can write the Hilbert-Poincar\'e series of $(R,W)$ as
\begin{equation}
\HP_{R} (t) =  \frac{Q_0(t)}{(1-t)^k} + \sum_{j \in T_{\delta}} \frac{Q_j^{(\delta)}(t)}{(1-\zeta^j t)^{\delta}} + \dots + \sum_{j \in T_1} \frac{Q_j^{(1)}(t)}{(1-\zeta^j t)}
\end{equation}
Thanks to remark \ref{remark1}, we have that the coefficients of the terms of degrees $\delta,\dots,k-1$ don't change and the result follows.
\end{proof}

Let us return to example \ref{ex2}. In that case, $\delta$ is equal to 3. By proposition \ref{prop link}, the coefficients of the terms of degree $\delta, \dots, k-1$, that is the 3th and 4th coefficient, don't change and this is exactly what happens in the example.

\begin{proposition}
Let $(R,W)$ and $P_R^W$ be as above. The $r$th coefficient of $P_R^W$, for $r=0,\dots,k-1$, has periodicity equal to \[
\delta_r := lcm \; (\; gcd (d_i)_{i \in I} \mid |I| = r+1, I \subseteq \{1,\dots,k\} )
\]
\end{proposition}

Formally, the  proposition asserts that if we denote the $r$th coefficient of $P_i(x)$ by $a_{ir}$, then $a_{jr} = a_{ir}$ when $j = i+\delta_r \mod d$.

\begin{proof}
Let $\delta = max\{ |I| \mid gcd(d_i)_{i \in I} \neq 1 \; \text{and} \; I \subseteq \{1,\dots, k\} \}$ be as usual. Then $\delta_r = 1$ for $\delta \leq r \leq k-1$ and, by proposition \ref{prop link}, the assertion follows.
Let $ r \leq \delta-1$. Recall that $P_i(x) = S_0(x) + S_1(x) \zeta^i + \dots + S_{d-1}(x)\zeta^{(d-1)i} $, where $S_j (x) = \sum_{h=0}^{n_j} a_{jh} \binom{x+\alpha_j-h-1}{\alpha_j-1} \zeta^{-jh}$. Note that each $S_j(x)$ doesn't depend on $i$ and its degree is equal to $\alpha_j-1$. Therefore the $r$th coefficient of $P_i(x)$ is the $r$th coefficient of
\[
S_0(x) + S_{j_1}(x) \zeta^{i \cdot j_1} + \dots + S_{j_v}(x) \zeta^{i \cdot j_v}
\]
with
\begin{equation} \label{j1 ... jv}
\{j_1,\dots,j_v\} = \bigcup_{s \geq r+1} T_s = \bigcup_{r+1 \leq s \leq \delta} T_s
\end{equation}
where the second equality holds because of $T_{\delta+1} = \dots = T_k = \emptyset$.
We recall that 
\begin{equation*}
T_{s} = \bigcup _{\substack{I \subseteq \{1,\dots,k\} , \mid I \mid = s\\ gcd(d_i)_{i \in I} \neq 1 }}\left\lbrace \frac{d}{gcd(d_i)_{i \in I}}, \dots, (gcd(d_i)_{i \in I} -1) \cdot \frac{d}{gcd(d_i)_{i \in I}} \right\rbrace
\end{equation*}
To prove that the $r$th coefficient of $S_0(x) + S_{j_1}(x) \zeta^{i \cdot j_1} + \dots + S_{j_v}(x) \zeta^{i \cdot j_v}$ has period $\delta_r$ we use the induction on number $N$ of non-empty sets $T_s$ in Equation (\ref{j1 ... jv}).
Let $N=1$, then $r= \delta -1$. We observe that $\zeta^{\delta_r \cdot j_t} = 1$ for all $t=1,\dots, v$, (in fact, $j_t$ can be written as $a \cdot \frac{d}{gcd(d_i)_{i \in I}}$ for some integer $ 1 \leq a \leq (gcd(d_i)_{i \in I}-1)$, with $|I| = r+1$; since $gcd(d_i)_{i \in I}$ divides $\delta_r$, then $d$ divides $\delta_r \cdot j_t $). So we have
\begin{multline*}
S_0(x) + S_{j_1}(x) \zeta^{i \cdot j_1} + \dots +S_{j_u}(x)\zeta^{i \cdot j_v} =\\
 S_0(x) + S_{j_1}(x) \zeta^{(i+\delta_r) \cdot j_1} + \dots + S_{j_v}(x)\zeta^{(i+\delta_r) \cdot j_v}
\end{multline*}
and we are done.
Let $N \geq 2$, then $r \leq \delta -2$, and let us suppose that the thesis is true for $\bigcup_{r' \leq s \leq \delta} T_s$, with $r'  > \delta -N$, which means that the $r'$th coefficient has period equal to $\delta_{r'}$. 
We can order the set $\{j_1,\dots,j_u, j_{u+1},\dots,j_v\} $ in such a way that
\begin{itemize}
\item $j_1, \dots, j_u \in T_s$, with $s > r+1$, and $j_i \notin T_{r+1}$, for all $i=1,\dots, u$;
\item $j_{u+1}, \dots, j_v \in T_{r+1}$.
\end{itemize}
Let $\gamma:= lcm(gcd(d_i)_{i \in I} \mid |I| >r+1)$, then by induction hypothesis we have
\begin{multline*}
S_0(x) + S_{j_1}(x) \zeta^{i \cdot j_1} + \dots +S_{j_u}(x)\zeta^{i \cdot j_u} =\\
 S_0(x) + S_{j_1}(x) \zeta^{(i+\gamma) \cdot j_1} + \dots + S_{j_u}(x)\zeta^{(i+\gamma) \cdot j_u}
\end{multline*}
Since $\gamma$ divides $\delta_r$, it follows 
\begin{multline*}
S_0(x) + S_{j_1}(x) \zeta^{i \cdot j_1} + \dots +S_{j_u}(x)\zeta^{i \cdot j_u} =\\
 S_0(x) + S_{j_1}(x) \zeta^{(i+\delta_r) \cdot j_1} + \dots + S_{j_u}(x)\zeta^{(i+\delta_r) \cdot j_u}
\end{multline*}
In addition, since $\zeta^{ \delta_r \cdot j_t} = 1$ for all $t= u+1, \dots, v$, we have
\begin{multline*}
S_{j_{u+1}}(x) \zeta^{i \cdot j_{u+1}} + \dots +S_{j_v}(x)\zeta^{i \cdot j_v} =\\
S_{j_{u+1}}(x) \zeta^{(i+\delta_r) \cdot j_{u+1}} + \dots + S_{j_v}(x)\zeta^{(i+\delta_r) \cdot j_v}
\end{multline*}
and then, by putting together the two equations above, we are done.\\
\end{proof}

\begin{remark}
In some particular cases, for example when $k >> \delta$ or if the periods $\delta_r$ are quite small, knowing the period of the coefficients could be exploited to reduce the cost of computation for the Hilbert quasi-polynomials. 
\end{remark}
\

Return to example \ref{ex2}. We already know that $\delta=3$, then $\delta_r=1$ for $r=3,4$. Consider the other cases: for $r=2$, $\delta_r$ is exactly equal to 2 and for $r=1$, $\delta_r$ is exactly equal to 6.

\subsection{Formulas for the other coefficients of the Hilbert quasi-polynomial}
\sectionmark{Formulas for the other coefficients}
In this subsection, we give formulas to recover the $(k-2)$th and $(k-3)$th coefficient of the Hilbert quasi-polynomial from the weights $d_1, \dots, d_k$, and we illustrate how recover the others.

\medskip

We have already shown that if $\delta \leq k-2$ ($\delta \leq k-3$ respectively) then the $(k-2)$th ($(k-3)$th respectively) coefficient is the same for all $P_i$, $i=0,\dots,d-1$. For this reason, we suppose that $\delta$ is small enough to guarantee that the considered coefficient is equal for all $P_i$.

\begin{proposition} \label{k-2 coeff}
Let $(R,W)$ and $P_R^W$ be as above. If $\delta \leq k-2$, we denote the $(k-2)$th coefficient of $P_i$, for all $i=0, \dots,d-1$, by $c_{k-2}$. Then, it holds
\[
c_{k-2} = \frac{\sum_{i=1}^k d_i}{2 \cdot (k-2)! \cdot \prod_{i=1}^k d_i}
\]
\end{proposition}
\

\begin{proof}
Let $P_i (x) = S_0(x) + S_1(x)\zeta^i + \dots + S_{d-1}(x)\zeta^{i \cdot (d-1)}$ be as usual, where
\[
S_j(n) = \sum_{h=0}^{n_j} a_{jh} \binom{n+\alpha_j -h-1}{\alpha_j -1} \zeta^{-jh}.
\]
Since $\delta \leq k-2$ and by remark \ref{remark1}, the $(k-2)$th coefficient of $P_i$ is given by the $(k-2)$th coefficient of $S_0(n)$. Let $a_i := a_{0i}$, then we have
\[
S_0(n) = a_0 \binom{n+k-1}{k-1} + a_1\binom{n+k-2}{k-1} + \dots + a_{k-1} \binom{n}{k-1} + a_k \binom{n-1}{k-1} =
\]
\[
 = a_0 \cdot \frac{(n+k-1)\cdots (n+1)}{(k-1)!} + a_1 \cdot \frac{(n+k-2) \cdots (n)}{(k-1)!}  +
\]
\begin{equation}\label{eq S_0}
+ \dots +a_{k-1} \cdot \frac{(n)\cdots (n-k+2)}{(k-1)!} + a_k \cdot \frac{(n-1)\cdots (n-k+1)}{(k-1)!} =  \sum_{i=0}^{k-1} c_{i}n^{i}
\end{equation}

By this equation, we get

\[
c_{k-2} = a_0 \left[ \frac{(k-1)+(k-2)+ \dots +1}{(k-1)!} \right]
+ a_1 \left[ \frac{(k-2)+(k-3)+ \dots +1}{(k-1)!} \right] +
\]
\[
+ \dots + a_k \left[ \frac{(-1)+(-2)+ \dots + (-k+1)}{(k-1)!}  \right] =
\]
\[
= \frac{a_0}{(k-1)!} \cdot \frac{k(k-1)}{2} + \frac{a_1}{(k-1)!} \cdot \frac{(k-1)(k-2)}{2} + \cdots + \frac{a_k}{(k-1)!} \cdot \frac{(-k)(k-1)}{2}= 
\]
\[
= \frac{1}{2(k-2)!} \left[ k a_0 +  (k-2) a_1 + \dots + (k - 2k)a_k \right] =
\]
\begin{equation}\label{c_k-2}
= \frac{1}{2(k-2)!} \left[ k \sum_{i=0}^k a_i - 2\sum_{i=0}^k (i \cdot a_i) \right]
\end{equation}

We observe that 
\begin{equation}\label{Q_0 e Q'_0}
\sum_{i=0}^k a_i = Q_0(1) \quad \text{and} \quad \sum_{i=0}^k i \cdot a_i = Q_0'(1)
\end{equation}
Since $Q_0(t) = A_1(1-t)^{k-1} + A_2(1-t)^{k-2}+ \dots +A_{k-1}(1-t) + A_k $, by evaluating $Q_0(t)$ and $Q_0' (t)$ at $t=1$ and by putting together Equation (\ref{c_k-2}) and Equation (\ref{Q_0 e Q'_0}), we have

\[
c_{k-2} = \frac{1}{2(k-2)!} \left[ k \cdot Q_0(1) - 2 \cdot Q_0'(1) \right] = \frac{1}{2(k-2)!} \left[ k \cdot A_k + 2 \cdot A_{k-1} \right].
\]
By the following lemma and by recalling that $A_k = \frac{1}{\prod_{i=1}^k d_i}$, the proposition is proved.

\end{proof}

\begin{lemma}
The constant $A_{k-1}$ which appears in Equation (\ref{eq2 HP}) of the expression of $Q_0(t)$ is such that
\[
A_{k-1} = \frac{\sum_{i=1}^k d_i -k}{2 \cdot \prod_{i=1}^k d_i}
\]
\end{lemma}
\

\begin{proof}
In order to obtain an expression for $A_{k-1}$, evaluating at $t=1$ the derivative of the Equation (\ref{1 = ...}) is a good strategy. Before doing it, we observe that the evaluation at $t=1$ of the derivative of the equation gives 0 except for the terms in the first summation for $t=k-1$ and $t=k$. For this reason, we'll consider only these addends and we'll calculate only their derivatives.
\begin{equation*}
0 = \frac{\partial}{\partial t} \left[ (A_{k-1} (1-t) + A_k) \cdot \prod_{j \in T_\delta} (1-\zeta^jt)^{\delta} \cdots \prod_{j \in T_1} (1-\zeta^jt) \right]_{\left. \middle |_{\text{\phantom{\begin{LARGE}
L
\end{LARGE}}\!\!\!\!\!\!\!\!\!\!\!\!\!\!\!\!\begin{small}\emph{t}=1\end{small}}} \right.}
\end{equation*}
\begin{equation*}
\begin{split}
0 &= -A_{k-1} \cdot \prod_{j \in T_\delta} (1-\zeta^j)^{\delta} \cdots \prod_{j \in T_1} (1-\zeta^j)+ \\
&+ A_k \cdot \frac{\partial}{\partial t} \left[ \prod_{j \in T_\delta} (1-\zeta^jt)^{\delta} \cdots \prod_{j \in T_1} (1-\zeta^jt) \right]_{\left. \middle |_{\text{\phantom{\begin{LARGE}
L
\end{LARGE}}\!\!\!\!\!\!\!\!\!\!\!\!\!\!\!\!\begin{small}\emph{t}=1\end{small}}} \right.}\\
\end{split}
\end{equation*}
\newline
We already know that
\[
A_k = \frac{1}{\prod_{i=1}^k d_i} \qquad \text{and} \qquad \prod_{j \in T_\delta} (1-\zeta^j)^{\delta} \cdots \prod_{j \in T_1} (1-\zeta^j) = \prod_{i=1}^k d_i.
\]
So, we have 
\begin{equation}\label{eq A_k-1}
A_{k-1} = \frac{1}{(\prod_{i=1}^k d_i)^2} \cdot \frac{\partial}{\partial t} \left[ \prod_{j \in T_\delta} (1-\zeta^jt)^{\delta} \cdots \prod_{j \in T_1} (1-\zeta^jt) \right]_{\left. \middle |_{\text{\phantom{\begin{LARGE}
L
\end{LARGE}}\!\!\!\!\!\!\!\!\!\!\!\!\!\!\!\!\begin{small}\emph{t}=1\end{small}}} \right.}
\end{equation}
and it remains to calculate $ \frac{\partial}{\partial t} \left[ \prod_{j \in T_\delta} (1-\zeta^jt)^{\delta} \cdots \prod_{j \in T_1} (1-\zeta^jt) \right]$ and evaluate it at $t=1$.

Consider the following equations
\[
\frac{g(t)}{(1-t)^k}= \left[ \prod_{j \in T_\delta} (1-\zeta^jt)^{\delta}\right] \cdots \left[ \prod_{j \in T_1} (1-\zeta^jt) \right]
\]
\[
\frac{g(t)}{(1-t)^k}= \prod_{j=1}^k \frac{1-t^{d_j}}{1-t} = \prod_{j=1}^k \left(1+t+\dots+t^{d_j -1}\right)
\]
Then, it holds
\begin{equation*}
\begin{split}
&\frac{\partial}{\partial t} \left[ \prod_{j \in T_\delta} (1-\zeta^jt)^{\delta} \cdots \prod_{j \in T_1} (1-\zeta^jt) \right]  = \frac{\partial}{\partial t} \left[ \prod_{j=1}^k \left(1+t+\dots+t^{d_j -1}\right) \right] = \\ &=\sum_{i=1}^k \left[ \prod_{j\neq i} \left(1+t+\dots+t^{d_j -1}\right) \right] \left(1+2t+\dots+(d_i-1) t^{d_i -2}\right)
\end{split}
\end{equation*}
Finally, we have
\begin{equation*}
\begin{split}
&\frac{\partial}{\partial t} \left[ \prod_{j \in T_\delta} (1-\zeta^jt)^{\delta} \cdots \prod_{j \in T_1} (1-\zeta^jt) \right]_{\left. \middle |_{\text{\phantom{\begin{LARGE}
L
\end{LARGE}}\!\!\!\!\!\!\!\!\!\!\!\!\!\!\!\!\begin{small}\emph{t}=1\end{small}}} \right.} = \sum_{i=1}^k \left[ \left( \prod_{j \neq i} d_j \right) \frac{d_i(d_i -1)}{2} \right] = \\
&= \prod_{i=1}^k d_i \cdot \left( \sum_{i=1}^k \frac{d_i -1}{2} \right) = \frac{1}{2} \cdot \prod_{i=1}^k d_i \cdot \left[ \sum_{i=1}^k d_i -k \right]
\end{split}
\end{equation*} 
and by substituting this expression in Equation (\ref{eq A_k-1}), we are done. 
\end{proof}

Now, we give a similar formula for the $(k-3)$th coefficient of $P_R^W$, omitting the proof for lack of space, and we refer to \cite{tesiMascia} for all the details.

\begin{proposition} \label{k-3 coeff}
Let $(R,W)$ and $P_R^W$ be as above. If $\delta \leq k-3$, we denote  by $c_{k-3}$ the $(k-3)$th coefficient of $P_i$, for all $i=0, \dots,d-1$. Then, it holds
\[
c_{k-3} = \frac{3 \left( \sum_{i=1}^k d_i \right)^2 - \sum_{i=1}^k d_i^2}{24 (k-3)! \prod_{i=1}^k d_i}
\]
\end{proposition}
\

So far we have given formulas for the leading coefficient, $(k-2)$th and $(k-3)$th coefficient of Hilbert quasi-polynomials. Unfortunately, we haven't found an explicit correlation between the formulas for these coefficients, meaning that we are not able to define a recurrence relation to obtain all coefficients of $P_R^W$. On the other hand, we can sketch a strategy to calculated formulas for any coefficient of Hilbert quasi-polynomials.

\smallskip

Let us see in detail the necessary steps to compute the $r$th coefficient of $P_R^W$, for $1 < r \leq k-1$, with the assumption that $\delta \leq r$.

First of all, let $P_i (x) = S_0(x) + S_1(x)\zeta^i + \dots + S_{d-1}(x)\zeta^{i \cdot (d-1)}$ as usual, where
\[
S_j(n) = \sum_{h=0}^{n_j} a_{jh} \binom{n+\alpha_j -h-1}{\alpha_j -1} \zeta^{-jh}.
\]
Since $\delta \leq r$ and by remark \ref{remark1}, the $r$th coefficient of $P_i$ is given exactly by the $r$th coefficient of $S_0(n)$, which we denote by $c_r$.
Let us recall the following useful expression for the Hilbert-Poincar\'e series
\begin{equation*}
\begin{split}
\HP_{R} (t) = \frac{1}{g(t)} = \frac{Q_0(t)}{(1-t)^k} + \sum_{j \in T_{\delta}} \frac{Q_j^{(\delta)}(t)}{(1-\zeta^j t)^{\delta}} + \dots + \sum_{j \in T_1} \frac{Q_j^{(1)}(t)}{(1-\zeta^j t)}\\
\end{split} 
\end{equation*}
where $Q_0(t) = A_1 (1-t)^{k-1} + A_2 (1-t)^{k-2} + \dots + A_k \in \Q[t]$.

The coefficient $c_r$ depends only on $k$ and $A_{r+1}, \dots, A_{k}$. So, we need to compute the constants $A_i$, for $i=r+1, \dots, k$, in order to get a suitable expression for $c_r$. To compute $A_i$, we evaluate at $t=1$ the ith derivative of the equation shown in (\ref{1 = ...}).
We observe that this evaluation gives 0 except for the terms in the first summation for $i \leq t \leq k$. For this reason, it needs to calculate only the derivatives of these addends.

\section{How to compute Hilbert quasi-polynomials}
In this section we present an algorithm for an efficient calculation of Hilbert quasi-polynomials. We have written Singular procedures to compute the Hilbert quasi-polynomial for rings $\K[x_1,\ldots,x_k]/I$.  These procedures can be downloaded from the website \url{www.dm.unipi.it/~caboara/Research/HilbertQP}

\subsection{The algorithm}
Let $(R/I, W)$ be as usual, we wish to compute its Hilbert quasi-polynomial $P_{R/I}^W :=\{ P_0, \dots, P_{d-1}\}$.
Since we know degree bounds for Hilbert quasi-polynomials, we can compute them by means of interpolation.

\smallskip
 
First off, let us consider $I = (0)$. Each $P_j$ has degree equal to $k-1$, so, given $j=0, \dots, d-1$, we want to calculate $P_j (x) := a_0 + a_1 x + \dots + a_{k-1} x^{k-1}$
such that
\[
P_j (n) = H_R^W(n) \quad \text{for all } n \geq k-1 \text { and such that } n \equiv j  \mod d
\]

Therefore, let $\beta := \min \{m \in \N \mid j +md \geq k-1\}$ and we consider the following $k$ positive integers congruent to $j + \beta d$ modulo $d$
\[
x_r := j + \beta d + rd, \quad \text{for } r = 0, \dots, k-1
\]
By construction, we have $P_j (x_r) = H_R^W (x_r)$, which means that the polynomial $P_j(x)$ interpolates the points $(x_r, H_R^W(x_r))$. 

Since we know the leading coefficient $c_{k-1}$, we can reduce the number of data points $x_r$. Actually, if $\delta \leq k-3$, we can also exploit the formulas for $c_{k-2}$ and $c_{k-3}$. In the latter case, we get a system of linear equations in the coefficients $a_i$, with $i=0, \dots, k-4$.  The system in matrix-vector form reads
\
\begin{equation}\label{system HP}
\begin{bmatrix}
1 & x_0 & \ldots  & x_0^{k-4} \\
1 & x_1 & \ldots   & x_1^{k-4} \\
\vdots & \vdots    & \vdots    & \vdots \\
1 & x_{k-4} & \ldots  & x_{k-4}^{k-4} 
\end{bmatrix}
\begin{bmatrix} a_0 \\ a_1 \\ \vdots \\ a_{k-4} \end{bmatrix}  =
\begin{bmatrix} H_R^W(x_0) - \sum_{i=k-3}^{k-1} c_{i} x_0^{i} \\ H_R^W(x_1) - \sum_{i=k-3}^{k-1} c_{i} x_1^{i} \\ \vdots \\ H_R^W(x_{k-4}) - \sum_{i=k-3}^{k-1} c_{i} x_{k-4}^{i}\end{bmatrix}
\end{equation}

We observe that this algorithm requires the computation of $k-3$ values of the Hilbert function, the construction of a Vandermonde matrix of dimension $k-3$ and its inversion. We have not yet shown how to calculate $H_R^W(x_r)$, for $r=0, \dots, k-3$.
 
\smallskip
 
Let $n \in \N$. The problem of calculating $H_R^W(n)$ is equivalent to the problem of determining the number of partitions of an integer into elements of a finite set $S := \{ d_1, \dots, d_k \}$, that is, the number of solutions in non-negative integers, $\alpha_1, \dots, \alpha_k$, of the equation
\[
\alpha_1 d_1 + \cdots + \alpha_k d_k =n
\]
This problem was solved in the nineteenth century (\cite{sylvester1882subvariants}, \cite{glaisher1909formulae}) and the solution is the coefficient of $x^n$ in
\begin{equation} \label{power series}
\left[ (1- x^{d_1}) \cdots (1-x^{d_k}) \right]^{-1}
\end{equation}

We are going to give an efficent method for getting the coefficient of $x^n$ in the power series expansion of Equation (\ref{power series}). We refer to \cite{lee1992power} for a in-depth analysis on the power series expansion of a rational function. Let 

\[
g(x) = \prod_{i=1}^k (1-\lambda_i x)^{\alpha_i} \qquad \text{and} \qquad f(x) = \prod_{i = k+1}^l (1- \lambda_i x)^{\alpha_i}
\]
be any polynomials with constant coefficient 1, where $\lambda_i$ are distinct and non-zero and the degree of $f(x)$ is less than that of $g(x)$.

\begin{lemma} \label{n b(n)}
Let 
\[
\frac{f(x)}{g(x)} = \sum_{n \geq 0} b(n) x^n
\]
the power series expansion of $f(x)/g(x)$. Then,
\[
nb(n) = \sum_{r=1}^n \left(\sum_{i=1} ^k \alpha_i \lambda_i^r - \sum_{i=k+1}^l \alpha_i \lambda_i^r \right) b(n-r)
\]
\end{lemma}

\begin{proof}
Since 
\[
\frac{\partial}{\partial x} \left[ log \left( \frac{f(x)}{g(x)} \right) \right] = \sum_{i=1}^k \frac{\alpha_i \lambda_i}{ 1 - \lambda_i x} - \sum_{i=k+1}^l \frac{\alpha_i \lambda_i}{ 1 - \lambda_i x}
\]

we have 

\[
\sum_{n\geq 0} nb(n)x^{n-1} = \left[ \sum_{ r \geq 1} \left (\sum_{i=1} ^k \alpha_i \lambda_i^r - \sum_{i=k+1}^l \alpha_i \lambda_i^r \right) x^{r-1} \right] \left[ \sum_{s \geq 0} b(s) x^s \right]
\]
\

and the result follows by picking out the coefficient of $x^{n-1}$ on the right.
\end{proof}

Let $\zeta := \zeta_d$ be a primitive $d$th root of unity. Since
\begin{equation*}
\prod_{i=1}^k (1-x^{d_i}) = (1-x)^k  \prod_{j \in T_k} (1-\zeta^j x)^k \cdots \prod_{j \in T_1} (1-\zeta^j x) 
\end{equation*}

we can apply Lemma (\ref{n b(n)}) with $g(x) :=\prod_{i=1}^k (1-x^{d_i}) $ and $f(x) = 1$. Since by our assumptions, the $d_i$ are coprime, then surely $T_k = \emptyset$. We define $\tilde{T_i} := T_i$ for all $i=1, \dots, k-1$ and $\tilde{T_k} := \{0\}$ and thus we obtain the following recursive formula for computing $H_R^W(n)$

\begin{equation}\label{H_R^W recursive}
H_{R}^W (n) = \frac{1}{n} \sum_{r=1}^n \left[ \sum_{i=1}^k i \left( \sum_{j \in \tilde{T_i}} \zeta^{jr} \right) \right] H_{R}^W(n-r)
\end{equation}

It follows that if we know $H_R^W(i)$ for all $i=1, \dots, n-1$, we can easily compute $H_R^W(n)$ by means of equation (\ref{H_R^W recursive}). 

\medskip
 
Given an equation $\alpha_1 d_1 + \dots + \alpha_k d_k = n$ the problem of counting the number of non-negative integer solutions $\alpha_1, \dots, \alpha_k$  could be solved also using brute force. But, given $n \in \N$, to compute $H_R^W(n)$ with brute force needs $O(n^k)$ operations, whereas the procedure which we have implemented has a quadratic cost in $n$, in fact it needs $O(k n^2)$ operations.

\medskip

Up to now we have shown how to calculate $P_R^W$. For computing a quasi-polynomial $P_{R/I}^W = \{P_0, \dots, P_{d-1}\}$, for any vector $W$ and any homogeneous ideal $I$ of $R$, the procedure computes first $P_R^{W'}$, where $W'$ is obtained by dividing $W$ by $gcd(d_1, \dots, d_k)$, and then it produces $P_{R/I}^W$ starting from $P_R^{W}$, using the relation between $P_R^{W'}$ and $P_{R/I}^W$ showed in propositions (\ref{prop6}).

\subsection{Some example}

\noindent
\begin{example}\label{ex2}

\noindent
Let $R= \Q[x_1, \dots, x_5]$ be the polynomial ring graded by the integer vector $W = [1,2,3,4,6]$ as in Example \ref{ex1}. We have already described $P_{R}^W$. Let us consider the ring $R/I$ with $I = (x_1^3, x_2 x_3)$. The Hilbert quasi-polynomial $P_{R/I}^W$ is given by
{
\begin{tabular}{l l l l l l l}
$P_0(x)$&$=$&$\textcolor{blue}{1/16}x^2$&$+$&$\textcolor{blue}{1/2}x$&$+$&$\textcolor{grad0}{1}$ \\
$P_1(x)$&$=$&$\textcolor{olive}{1/24}x^2$&$+$&$\textcolor{olive}{1/3}x$&$+$&$\textcolor{grad1}{5/8}$ \\
$P_2(x)$&$=$&$\textcolor{blue}{1/16}x^2$&$+$&$\textcolor{blue}{1/2}x$&$+$&$\textcolor{grad2}{3/4}$ \\
$P_3(x)$&$=$&$\textcolor{olive}{1/24}x^2$&$+$&$\textcolor{olive}{1/3}x$&$+$&$\textcolor{grad3}{5/8}$ \\
$P_4(x)$&$=$&$\textcolor{blue}{1/16}x^2$&$+$&$\textcolor{blue}{1/2}x$&$+$&$\textcolor{grad4}{1}$ \\
$P_5(x)$&$=$&$\textcolor{olive}{1/24}x^2$&$+$&$\textcolor{olive}{1/3}x$&$+$&$\textcolor{grad5}{7/24}$ \\
$P_6(x)$&$=$&$\textcolor{blue} {1/16}x^2$&$+$&$\textcolor{blue}{1/2}x$&$+$&$\textcolor{grad6}{3/4}$ \\
$P_7(x)$&$=$&$\textcolor{olive} {1/24}x^2$&$+$&$\textcolor{olive}{1/3}x$&$+$&$\textcolor{grad7}{5/8}$ \\
$P_8(x)$&$=$&$\textcolor{blue}{1/16}x^2$&$+$&$\textcolor{blue}{1/2}x$&$+$&$\textcolor{grad8}{1}$ \\
$P_9(x)$&$=$&$\textcolor{olive}{1/24}x^2$&$+$&$\textcolor{olive}{1/3}x$&$+$&$\textcolor{grad9}{5/8}$ \\
$P_{10}(x)$&$=$&$\textcolor{blue}{1/16}x^2$&$+$&$\textcolor{blue}{1/2}x$&$+$&$\textcolor{grad10}{3/4}$\\
$P_{11}(x)$&$=$&$\textcolor{olive}{1/24}x^2$&$+$&$\textcolor{olive}{1/3}x$&$+$&$\textcolor{grad11}{7/24}$ \\
\end{tabular}
}
\end{example}

The time of computation for all $P_R^W$ and $P_{R/I_i}^W$ is less than 1 second. 

\noindent
\begin{example}\label{ex3}
\noindent
Let $R := \Q[x,y,z,t,u]$ be graded by $W := [2,4,8,16,32]$. We have computed $P_R^W = \{P_0, \dots, P_{31}\}$ and also in this case the time of computation is less than 1 second. The odd indexed polynomials are equal to $0$, while the even indexed are
\begin{flushleft}
\begin{small}
\begin{tabular}{l l l}
$P_0(x)$&$=$&$\textcolor{BlueCarisma}{1/393216}x^4+\textcolor{teal}{1/3072}x^3+\textcolor{grad_0}{43/3072}x^2+\textcolor{grad_0}{11/48}x+\textcolor{grad0}1$\\[1pt]
$P_2(x)$&$=$&$\textcolor{BlueCarisma}{1/393216}x^4+\textcolor{olive}{5/16384}x^3+\textcolor{grad_1}{595/49152}x^2+\textcolor{grad_1}{725/4096}x+\textcolor{grad1}{4875/8192}$\\[1pt]
$P_4(x)$&$=$&$\textcolor{BlueCarisma}{1/393216}x^4+\textcolor{teal}{1/3072}x^3+\textcolor{grad_2}{169/12288}x^2+\textcolor{grad_2}{167/768}x+\textcolor{grad2}{455/512}$\\[1pt]
$P_6(x)$&$=$&$\textcolor{BlueCarisma}{1/393216}x^4+\textcolor{olive}{5/16384}x^3+\textcolor{grad_3}{583/49152}x^2+\textcolor{grad_3}{681/4096}x+\textcolor{grad3}{4147/8192}$\\[1pt]
$P_8(x)$&$=$&$\textcolor{BlueCarisma}{1/393216}x^4+\textcolor{teal}{1/3072}x^3+\textcolor{grad_0}{43/3072}x^2+\textcolor{grad_0}{11/48}x+\textcolor{grad4}{35/32}$\\[1pt]
$P_{10}(x)$&$=$&$\textcolor{BlueCarisma}{1/393216}x^4+\textcolor{olive}{5/16384}x^3+\textcolor{grad_1}{595/49152}x^2+\textcolor{grad_1}{725/4096}x+\textcolor{grad5}{5643/8192}$\\[1pt]
$P_{12}(x)$&$=$&$\textcolor{BlueCarisma}{1/393216}x^4+\textcolor{teal}{1/3072}x^3+\textcolor{grad_2}{169/12288}x^2+\textcolor{grad_2}{161/768}x+\textcolor{grad6}{455/512}$\\[1pt]
$P_{14}(x)$&$=$&$\textcolor{BlueCarisma}{1/393216}x^4+\textcolor{olive}{5/16384}x^3+\textcolor{grad_3}{583/49152}x^2+\textcolor{grad_3}{649/4096}x+\textcolor{grad7}{4275/8192}$\\[1pt]
$P_{16}(x)$&$=$&$\textcolor{BlueCarisma}{1/393216}x^4+\textcolor{teal}{1/3072}x^3+\textcolor{grad_0}{43/3072}x^2+\textcolor{grad_0}{11/48}x+\textcolor{grad8}{5/4}$\\[1pt]
$P_{18}(x)$&$=$&$\textcolor{BlueCarisma}{1/393216}x^4+\textcolor{olive}{5/16384}x^3+\textcolor{grad_1}{595/49152}x^2+\textcolor{grad_1}{725/4096}x+\textcolor{grad9}{6923/8192}$\\[1pt]
$P_{20}(x)$&$=$&$\textcolor{BlueCarisma}{1/393216}x^4+\textcolor{teal}{1/3072}x^3+\textcolor{grad_2}{169/12288}x^2+\textcolor{grad_2}{167/768}x+\textcolor{grad10}{583/512}$\\[1pt]
$P_{22}(x)$&$=$&$\textcolor{BlueCarisma}{1/393216}x^4+\textcolor{olive}{5/16384}x^3+\textcolor{grad_3}{583/49152}x^2+\textcolor{grad_3}{681/4096}x+\textcolor{grad11}{6195/8192}$\\[1pt]
$P_{24}(x)$&$=$&$\textcolor{BlueCarisma}{1/393216}x^4+\textcolor{teal}{1/3072}x^3+\textcolor{grad_0}{43/3072}x^2+\textcolor{grad_0}{11/48}x+\textcolor{grad12}{35/32}$\\[1pt]
$P_{26}(x)$&$=$&$\textcolor{BlueCarisma}{1/393216}x^4+\textcolor{olive}{5/16384}x^3+\textcolor{grad_1}{595/49152}x^2+\textcolor{grad_1}{725/4096}x+\textcolor{grad13}{5643/8192}$\\[1pt]
$P_{28}(x)$&$=$&$\textcolor{BlueCarisma}{1/393216}x^4+\textcolor{teal}{1/3072}x^3+\textcolor{grad_2}{169/12288}x^2+\textcolor{grad_2}{161/768}x+\textcolor{grad14}{327/512}$\\[1pt]
$P_{30}(x)$&$=$&$\textcolor{BlueCarisma}{1/393216}x^4+\textcolor{olive}{5/16384}x^3+\textcolor{grad_3}{583/49152}x^2+\textcolor{grad_3}{649/4096}x+\textcolor{grad15}{2227/8192}$\\[1pt]
\end{tabular}
\end{small}
\end{flushleft}
In this example, we have $k=5$, $\delta = 2$, $\delta_4 = 2$, $\delta_3 = 4$, $\delta_2= 8$, $\delta_1 = 16$ and $\delta_0 = 32$. We can easily check that $P_R^W$ satisfies all proprieties that we have illustrated in the previous sections.
\end{example}

\section{Conclusions and further work}

In this work, we have produced a partial characterization of Hilbert quasi-polynomials for  the
$\N_{+}^{k}$-graded polynomial rings. It should be of interest to complete this characterization, finding closed formulas for as many as possible coefficients of the Hilbert quasi-polynomial, periodic part included. As a side effect, this will allow us to write more efficient procedures for the computation of Hilbert quasi-polynomials. 

\smallskip

Actually, in section 2 we have sketched a strategy to recover coefficients for the fixed part, but a limit of this method is the difficulty of simplifying computations. Familiarity with combinatoric tools could be of great help to obtain formulas easy to read and to use, like those we have already found out.

\smallskip

To compute the $r$th coefficient of $P_R^W$, we have supposed $\delta \leq r$. We can ask what is the behaviour of the coefficients without any assumptions about $\delta$. In section 2, we have proved that they keep repeating with a certain period $\delta_r$. We are interesting in understanding if it is possible to find formulas also in the general case. We have analysed some numerical experiments, and we have observed that there is a sort of regularity in the expression of the coefficients. 
Describing completely the coefficients of $P_R^W$ requires considerable effort, but it is an interesting topic for future work.
Moreover, we want to extend our work to $\N_{+}^{k}$-graded quotient rings $\K[x_1,\ldots,x_k]/I$, starting with simple cases (e.g. when the ideal $I$ is generated by a regular sequence).



\nocite{*}
\bibliographystyle{alpha}
\bibliography{biblio}

\end{document}